\documentclass[11pt,reqno]{amsart}
\usepackage{amsmath,amsfonts,amssymb}

\allowdisplaybreaks
\numberwithin{equation}{section}

\newtheorem{theorem}{Theorem}
\newtheorem{lemma}{Lemma}
\theoremstyle{remark}
\newtheorem{remark}{Remark}

\newcommand{\del}{\partial}
\newcommand{\vep}{\varepsilon}
\newcommand{\bk}{\backslash}
\newcommand{\HH}{\mathbb{H}}
\newcommand{\D}{\mathbb{D}}

\newcommand{\Q}{\mathbb{Q}}
\newcommand{\C}{\mathbb{C}}
\newcommand{\Z}{\mathbb{Z}}
\usepackage{hyperref}
\hypersetup{colorlinks,citecolor=blue}

\begin{document}
\title[Local index theorem for orbifold Riemann surfaces]{Local index theorem for orbifold Riemann surfaces}
\author{Leon A. Takhtajan}
\address{Department of Mathematics,
Stony Brook University, Stony Brook, NY 11794 USA; Euler International Mathematical Institute, Pesochnaya Nab. 10, Saint Petersburg 197022 Russia}
\author{Peter Zograf}
\address{Steklov Mathematical Institute, Fontanka 27, Saint Petersburg 191023 Russia; Chebyshev Laboratory, Saint Petersburg State University, 14-th Line V.O. 29, Saint Petersburg 199178 Russia}
\thanks{Supported by RScF grant 16-11-10039}
\date{\today}
\begin{abstract}
We derive a local index theorem in Quillen's form for families of Cauchy-Riemann operators on orbifold Riemann surfaces (or Riemann orbisurfaces) 
that are quotients of the hyperbolic plane by the action of cofinite finitely generated Fuchsian groups. 
Each conical point (or a conjugacy class of primitive elliptic elements in the Fuchsian group) gives rise to an extra term in the local index theorem 
that is proportional to the symplectic form of a new K\"{a}hler metric on the moduli space of Riemann orbisurfaces.
We find a simple formula for a local K\"{a}hler potential of the elliptic metric and show that when the order of elliptic element becomes large, the elliptic metric converges to the cuspidal one corresponding to a puncture on the orbisurface (or a conjugacy class of primitive parabolic elements). We also give a simple example of a relation between the elliptic metric and special values of Selberg's zeta function.
\end{abstract}

\maketitle

\section{Introduction}
Quillen's local index theorem for families of Cauchy-Riemann operators \cite{Q} explicitly computes the first Chern form of the corresponding determinant line bundles equipped with Quillen's metric. The advantage of local formulas becomes apparent when the families parameter spaces are non-compact. In the language of algebraic geometry, Quillen's local index theorem is a manifestation of the ``strong'' Grothendieck-Riemann-Roch theorem that claims an isomorphism between metrized
holomorphic line bundles. 

The literature on Quillen's local index theorem is abundant, but mostly deals with families of smooth compact varieties. In this paper we derive a general local index theorem for families of Cauchy-Riemann operators on Riemann orbisurfaces, both compact and with punctures, that appear as quotients $X=\Gamma\backslash\HH$ of the hyperbolic plane 
$\HH$ by the action of finitely generated cofinite Fuchsian groups $\Gamma$. The main result (cf. Theorem \ref{theorem 2}) is the following formula on the moduli space associated with the group $\Gamma$:
\pagebreak
\begin{align*}
&c_1(\lambda_{k}, ||\cdot||_{k}^Q)=\dfrac{6k^2-6k+1}{12\pi^2}\omega_{\mathrm{WP}}-\dfrac{1}{9}\omega_{\rm cusp}\\
&+\dfrac{1}{4\pi}\sum_{j=1}^l \left(2(k-1)B_{1}\left(\left\{\frac{k-1}{m_{j}}\right\}\right)-m_{j} B_2\left(\left\{\dfrac{k-1}{m_j}\right\}\right)+\dfrac{6k-5}{6m_j}\right)\omega_j^{\rm ell}.
\end{align*}
Here $c_1(\lambda_{k}, ||\cdot||_{k}^Q)$ is the first Chern form of the determinant line bundle $\lambda_{k}$ of the vector bundle of square integrable meromorphic $k$-differentials, $k\geq 1$, on 
$X=\Gamma\backslash\HH$ equipped with the Quillen's metric, $\omega_{\mathrm{WP}}$ is a symplectic  form of the Weil-Petersson metric on the moduli space, $\omega_{\rm{cusp}}$ is a symplectic form of the cuspidal metric (also known as Takhtajan-Zograf metric), $\omega_j^{\rm ell}$ is the symplectic form of a K\"{a}hler metric associated with elliptic fixpoints, $B_{1}(x)=x-\frac{1}{2}$ and $B_2(x)=x^2-x+\frac{1}{6}$ are Bernoulli polynomials, and $\{x\}$ is the fractional part of $x\in\Q$. We refer the reader to Sections \ref{2.1}--\ref{2.3} and \ref{3.2} for the definitions and precise statements. Note that the above formula is equivalent to formula \eqref{main} for $k\leq 0$ because the Hermitian line bundles $\lambda_k$ and $\lambda_{1-k}$ on the moduli space are isometrically isomorphic, see Remark \ref{isom}. 

Note that the case of smooth punctured Riemann surfaces was treated by us much earlier in
\cite{TZ}, and now we add conical points into consideration. The motivation to study families of Riemann orbisurfaces comes from various areas of mathematics and theoretical physics -- from Arakelov geometry \cite{FP} to the theory of quantum Hall effect \cite{C}. In particular, the paper \cite{FP} 
that establishes the Riemann-Roch type isometry for non-compact orbisurfaces as Deligne isomorphism of metrized $\mathbb{Q}$-line bundles stimulated us to  extend the results of \cite{TZ} to the orbisurface setting. 

The paper is organized as follows. Section \ref{2} contains the necessary background material. In Section \ref{3} we prove the local index theorem for families of $\bar\del$-operators on Riemann orbisurfaces that are factors of the hyperbolic plane by the action of finitely generated cofinite Fuchsian groups. Specifically, 
we show that the contribution to the local index formula from elliptic elements of Fuchsian groups is given by the symplectic form of a K\"{a}hler metric on the moduli space of orbisurfaces.
 Since the cases of smooth (both compact and punctured) Riemann surfaces have been well understood by us quite a while ago  \cite{ZT,TZ}, in Section \ref{3.2} we mainly emphasize the computation of the contribution from conical points corresponding to elliptic elements. In Section \ref{potential} we find a simple formula for a local K\"{a}hler potential of the elliptic metric, and in Section \ref{relation} we show that in the limit when the order of the elliptic element 
tends to $\infty$ the elliptic metric coincides with the corresponding cusp metric. Finally, in Section \ref{zeta values} we give a simple example of a relation between the elliptic metric and special values of Selberg zeta 
function for Fuchsian groups of signature (0;1;2,2,2).

\subsection*{Acknowledgments.} We thank G.~Freixas i Montplet for showing us a preliminary version of \cite{FP} and for stimulating discussions.
We are especially grateful to Lee-Peng Teo, who carefully read the manuscript and pointed out a number of misprints to us. Most importantly, she noticed that the term containing the first Bernoulli polynomial $B_{1}(x)$ was missing from the main formula in the original version of the paper, and recently proposed an independent approach to its derivation, see \cite{LPT}. 

\section{Preliminaries} \label{2}
\subsection{Hyperbolic plane and Fuchsian groups}\label{2.1}
We will use two models of the Lobachevsky (hyperbolic) plane: the upper half-plane $\HH=\{z\in\mathbb{C}\,\big|\,{\rm Im}z>0\}$ 
with the metric $\dfrac{|dz|^2}{({\rm Im}z)^2}$, 
and the Poincar\'{e} unit disk $\D=\{u\in\mathbb{C}\,\big|\,|u|<1\}$ with the metric $\dfrac{4|du|^2}{(1-|u|^2)^2}$. The biholomorphic isometry between the two models is given by the linear fractional transformation $u=\dfrac{z-z_0}{z-\bar{z}_0}$ for any $z_0\in \HH$.

A Fuchsian group $\Gamma$ of the first kind is a finitely generated cofinite discrete subgroup of $\mathrm{PSL}(2,\mathbb{R})$ acting on $\HH$ (it can also be considered as a subgroup of $\mathrm{PSU}(1,1)$ acting on $\D$). 
Such $\Gamma$ has a standard presentation with $2g$ hyperbolic generators $A_1,B_1,\ldots,A_g,B_g$, $n$ parabolic generators $S_1,\ldots,S_n$ and $l$ elliptic generators $T_1,\ldots,T_l$ of orders $m_1,\ldots,m_l$ satisfying the relations
\begin{align*}
&A_1 B_1A_1^{-1}B_1^{-1}\ldots A_g B_g A_g^{-1}B_g^{-1} S_1\ldots S_n T_1\ldots T_l=I,\\ 
&T_i^{m_i}=I,\;\; i=1,\ldots,l,
\end{align*}
where $I$ is the identity element. The set $(g;n;m_1,\ldots,m_l)$, where $2\leq m_{1}\leq\cdots\leq m_{l}$, is called the signature of $\Gamma$, and we will always assume that
$$2g-2+n+\sum_{i=1}^l \left(1-\dfrac{1}{m_i}\right)>0\,.$$

We will be interested in orbifolds $X=\Gamma\backslash\HH$ (or $X=\Gamma\backslash\D$, if we treat $\Gamma$ as acting on $\D$) for Fuchsian groups $\Gamma$ of the first kind. Such an orbifold is a Riemann surface of genus $g$ with $n$ punctures and $l$ conical points of angles $\dfrac{2\pi}{m_1},\ldots,\dfrac{2\pi}{m_l}$. By a $(p,q)$-differential on the orbifold Riemann surface $X=\Gamma\backslash\HH$ we understand a smooth function $\phi$ on $\HH$ that transforms according to the rule $\phi(\gamma z)\gamma'(z)^p\overline{\gamma'(z)}^q=\phi(z)$. The space of harmonic $(p,q)$-differentials, square integrable with respect to the hyperbolic metric on $X=\Gamma\backslash\HH$, we denote by $\Omega^{p,q}(X)$. 
The dimension of the space of square integrable meromorphic (with poles at punctures and conical points) $k$-differentials on $X$, or cusp forms of weight $2k$ for $\Gamma$, is given by Riemann-Roch formula for orbifolds:
\begin{align*}
\dim\Omega^{k,0}(X)=
\begin{cases} (2k-1)(g-1)+(k-1)n+\sum\limits_{i=1}^l \left[k\left(1-\dfrac{1}{m_i}\right)\right],\;\; k>1,\\ g, \quad k=1,\\ 1, \quad k=0,\\ 0, \quad k<0,
\end{cases}
\end{align*}
where $[r]$ denotes the integer part of $r\in\Q$ (see  \cite[Theorem 2.24]{Sh}). In particular, 
$$\dim\Omega^{2,0}(X)=3g-3+n+l.$$

The elements of the space $\Omega^{-1,1}(X)$ are called harmonic Beltrami differentials and play an important role in the deformation theory of Fuchsian groups, see Sect. \ref{2.3}. To study the behavior of harmonic Beltrami differentials at the elliptic fixpoints we use the unit disk model. Take $\mu\in\Omega^{-1,1}(X)$ and let $T\in\Gamma$ be an elliptic element of order $m$ with fixpoint $z_0\in H$. The pushforward of $T$ to $\D$ by means of the map $u=\dfrac{z-z_0}{z-\bar{z}_0}$ is just the multiplication by 
$\omega=e^{2\pi\sqrt{-1}/m}$, the $m$-th primitive root of unity. The pushforward of $\mu$ to $\D$ (that, slightly abusing notation, we will denote by the same symbol) develops into a power series of the form
$$
\mu(u)=\dfrac{(1-|u|^{2})^{2}}{4}\sum_{n=2}^{\infty}\bar{a}_{n}\bar{u}^{n-2}.
$$
Moreover, since $\mu(\omega u)=\mu(u)\omega^{-2}$
we have $a_{n}=0$ unless $n\equiv 0\!\!\mod m$, so that
\begin{align}
\mu(u)=\dfrac{(1-|u|^{2})^{2}}{4}\sum_{j=1}^{\infty}\bar{a}_{jm}\bar{u}^{jm-2}\,.\label{mu-disk}
\end{align}
In particular, $\mu(0)=0$ for $m>2$ and $\dfrac{\del\mu}{\del u}(0)=0$ for $m=2$.

As in \cite{TZ}, for $\mu,\nu\in\Omega^{-1,1}(X)$ we put $f_{\mu\bar\nu}=(\Delta_{0}+\tfrac{1}{2})^{-1}(\mu\bar\nu)$, where
$$
\Delta_{0}=-y^2\dfrac{\del^{2}}{\del z \del\bar{z}},\quad y={\rm Im}\,z\;,
$$
is the Laplace operator (or rather $1/4$ of the Laplacian) in the hyperbolic metric acting on $\Omega^{0,0}(X)$.
The function $f_{\mu\bar\nu}(u)$ is regular on $\D$ and satisfies
$$f_{\mu\bar\nu}(\omega u)=f_{\mu\bar\nu}(u).$$

The following result is analogues to Lemma 2 in \cite{TZ} and describes the behavior of $f_{\mu\bar\nu}(u)$ at $u=0$. We will use polar coordinates on $\D$ such that $u=re^{\sqrt{-1}\theta}$.

\begin{lemma}\label{f-0} Let  
\begin{align}
f_{\mu\bar\nu}(u)=\sum_{j=-\infty}^{\infty}f_{jm}(r)e^{\sqrt{-1}jm\theta}\label{fou}
\end{align}
be the Fourier series of the function $f_{\mu\bar\nu}(u)$ on $\D$. Then
\begin{itemize}
\item[(i)] 
$f_{0}(r)=c_{0}+c_{2}r^{2}+O(r^{4})$ as $r\rightarrow 0$, where
\begin{align}
c_{2}=\begin{cases} 2c_{0}, & m>2,\\
2c_{0}-4\mu(0)\bar\nu(0), & m=2.
\end{cases}\label{f_0}
\end{align}
\item[(ii)] 
$f_{n}(r)=O(r^{|n|})$ as $r\rightarrow 0$;
\item[(iii)] For the constant term $c_0=f_0(0)$ we have
$$
c_{0}=\int\limits_{X}G(0,u)\mu(u)\overline{\nu(u)}d\rho(u)\,,
$$
where $G(u,v)$ is the integral kernel of $(\Delta_{0}+\tfrac{1}{2})^{-1}$ on $X=\Gamma\backslash\D$, and $d\rho(u)=\dfrac{2\sqrt{-1}}{(1-|u|^2)^2}\,du\wedge d\bar{u}$.
\end{itemize}
\end{lemma}

\begin{proof}
Since $f_{\mu\bar\nu}(u)$ is a regular solution of the equation
$(\Delta_{0}+\tfrac{1}{2})f=\mu\bar\nu$ at $u=0$,
we have in polar coordinates
\begin{align*}
&-\dfrac{(1-r^{2})^{2}}{16}\left(\dfrac{\del^{2}f}{\del r^{2}}+\dfrac{1}{r}\dfrac{\del f}{\del r}+\dfrac{1}{r^{2}}\dfrac{\del^{2}f}{\del\theta^{2}}\right)(r,\theta)+\dfrac{1}{2}f(r,\theta)=\\
&=\dfrac{(1-r^{2})^{4}}{16}\sum_{i=1}^{\infty}\sum_{j=1}^{\infty}\bar{a}_{im}b_{jm}r^{(i+j)m-4}e^{\sqrt{-1}(j-i)m\theta},
\end{align*}
where we used \eqref{mu-disk} for $\mu(u)$ and the analogous expansion 
\begin{align}
\nu(u)=\dfrac{(1-|u|^{2})^{2}}{4}\sum_{j=1}^{\infty}\bar{b}_{jm}\bar{u}^{jm-2}.\label{nu-disk}
\end{align}
for $\nu(u)$. Then for the term $f_0(r)$ of the Fourier series \eqref{fou} we have the differential equation
\begin{align*}
-\dfrac{(1-r^{2})^{2}}{16}\left(\dfrac{d^{2}f_{0}(r)}{dr^{2}}+\dfrac{1}{r}\dfrac{df_{0}(r)}{dr}\right)+\dfrac{1}{2}f_{0}(r)
=\dfrac{(1-r^{2})^{4}}{16}\sum_{j=1}^{\infty}\bar{a}_{jm}b_{jm}r^{2jm-4}.
\end{align*}
From here we get that $f_{0}(r)=c_{0}+c_{2}r^{2}+O(r^{4})$ as $r\rightarrow 0$,
where
\begin{align*}
c_{2}=\begin{cases} 2c_{0}, & m>2,\\
2c_{0}-4\mu(0)\bar\nu(0), & m=2.
\end{cases}
\end{align*}
For the coefficients $f_n(r)$ with $n\neq 0$ we have
\begin{align*}
&-\dfrac{(1-r^{2})^{2}}{16}\left(\dfrac{d^{2}f_{n}}{dr^{2}}+\dfrac{1}{r}\dfrac{df_{n}}{dr}-\dfrac{n^{2}f_{n}(r)}{r^{2}}\right)+\dfrac{1}{2}f_{n}(r)=\\
&=\dfrac{(1-r^{2})^{4}}{16}\sum_{j=1}^{\infty}\bar{a}_{jm}b_{jm+n}r^{2jm+n-4}\,,
\end{align*}
so that $f_{n}(r)=O(r^{|n|})$ as $r\rightarrow 0$. This proves parts (i) and (ii) of the lemma, from where it follows that $c_{0}=f_{\mu\bar\nu}(0)$. To prove part (iii) it is sufficient to observe that 
$$f_{\mu\nu}(0)=\int\limits_{X}G(0,u)\mu(u)\overline{\nu(u)}d\rho(u)\,.$$
\end{proof}

\subsection{Laplacians on Riemann orbisurfaces}
Let us now switch to the properties of the Laplace operators on the hyperbolic orbifold $X=\Gamma\backslash\HH$, where $\Gamma$ is a Fuchsian group of the first kind. 
Here we give only a brief sketch, and the details can be found in \cite{ZT}, \cite{TZ}.
Denote by $\mathcal{H}^{p,q}$ the Hilbert space of $(p,q)$-differentials on $X$, and let $\bar\del_k:\mathcal{H}^{k,0}\rightarrow\mathcal{H}^{k,1}$
be the Cauchy-Riemann operator acting on $(k,0)$-differentials (in terms of the coordinate $z$ on $\HH$ we have $\bar\del_k=\del/\del\bar{z}$). Denote by $\bar\del_k^*:\mathcal{H}^{k,1}\rightarrow\mathcal{H}^{k,0}$ the formal adjoint to
$\bar\del_k$ and define the Laplace operator acting on $(k,0)$-differentials on $X$ by the formula $\Delta_k=\bar\del_k^*\bar\del_k$.

We denote by $Q_k(z,z';s)$ the integral kernel of $\left(\Delta_k+\tfrac{(s-1)(s-2k)}{4}I\right)^{-1}$ on the entire upper half-plane $\HH$ (where $I$ is the identity operator in the Hilbert space of $k$-differentials on $\HH$).
The kernel $Q_k(z,z';s)$ is smooth for $z\neq z'$ and has an important property that 
$Q_k(z,z';s)=Q_k(\sigma z,\sigma z';s)\sigma'(z)^k\overline{\sigma'(z')^k}$ for any $\sigma\in PSL(2,\mathbb{R})$. For $k\geq 0$ and $s=1$ we have the explicit formula
\begin{align}
\dfrac{\del}{\del z}y^{-2k}\dfrac{\del}{\del z'}Q_{-k}(z,z';1)=-\dfrac{1}{\pi}\cdot\dfrac{1}{(z-z')^2}\left(\dfrac{z'-\bar z'}{\bar{z}-z'}\right)^{2k}\;,\label{Q}
\end{align}
where $y={\rm Im}\,z$. 

Furthermore, denote by $G_{k}(z,z';s)$ the integral kernel of the resolvent $\left(\Delta_k+\tfrac{(s-1)(s-2k)}{4}I\right)^{-1}$ of $\Delta_k$ on $X=\Gamma\backslash\HH$ (where $I$ is the identity operator in the Hilbert space $\mathcal{H}^{k,0}$). 
For $k<0$ and $s=1$ the Green's function $G_{k}(z,z';s)$ is a smooth function on $X\times X$ away from the diagonal (i.~e. for $z\neq z'$).
For $k=0$ we have the following Laurent expansion near $s=1$:
\begin{align}
G_{0}(z,z';s)=\dfrac{4}{|X|}\cdot\dfrac{1}{s(s-1)}+G_0(z,z')+ O(s-1)\label{G_0}
\end{align}
as $s\to 1$, where $|X|=2\pi\left(2g-2+n+\sum_{i=1}^l(1-1/m_i)\right)$ is the hyperbolic area of $X=\Gamma\backslash\HH$.
Then for any integer $k\geq 0$ we have
\begin{align}
\dfrac{\del}{\del z}y^{-2k}\dfrac{\del}{\del z'}\,G_{-k}(z,z';1)
=-\dfrac{1}{\pi}\sum_{\gamma\in\Gamma}\dfrac{1}{(z-\gamma z')^2}\left(\dfrac{\gamma\bar{z}'-\gamma z'}{\bar{z}-\gamma z'}\right)^{2k}
\gamma'(z')\gamma'(\bar{z}')^{-k}\;.\label{G}
\end{align}
This series converges absolutely and uniformly on compact sets for $z\neq\gamma z',\;\gamma\in\Gamma$. 

We now recall the definition of the Selberg zeta function. Let $\Gamma$ be a Fuchsian group of the first kind, and let $\chi:\Gamma\to U(1)$ be a unitary character. Put
\begin{align}
Z(s,\Gamma,\chi)=\prod_{\{\gamma\}}\, \prod_{i=0}^\infty \left(1-\chi(\gamma) N(\gamma)^{-s-i}\right)\;,\label{zeta}
\end{align}
where $\{\gamma\}$ runs over the set of classes of conjugate hyperbolic elements of $\Gamma$, and $N(\gamma)$ is the norm of $\gamma$ defined by the conditions
$N(\gamma)+1/N(\gamma)=|{\rm tr}\,\gamma|,\; N(\gamma)>1$ (in other words, $\log N(\gamma)$ is the length of the closed geodesic in the free homotopy class associated with $\gamma$). The product \eqref{zeta} converges absolutely for ${\rm Re}\,s>1$ and admits a meromorphic continuation to the complex $s$-plane.

Except for the last section, in what follows we will always assume that $\chi\equiv 1$ and will denote $Z(s,\Gamma,1)$ simply by $Z(s)$. The Selberg trace formula relates $Z(s)$ to the spectrum of the Laplacians on $\Gamma\backslash\HH$, and it is natural (cf. \cite{DPh}) to define the regularized determinants of the operators $\Delta_{-k}$ by the formula
\begin{align}
\det\Delta_{-k}=\begin{cases} Z'(1), & k=0,\\Z(k+1), & k\geq 1,\end{cases}\label{det}
\end{align}
(note that $Z(s)$ has a simple zero at $s=1$). 

\subsection{Deformation theory} \label{2.3}
We proceed with the basics of the deformation theory of Fuchsian groups. Let $\Gamma$ be a Fuchsian group of the first kind of signature $(g;n;m_1,\ldots,m_l)$. Consider the space of quasiconformal mappings of the upper half-plane $\HH$ that fix 0, 1 and $\infty$. Two quasiconformal mappings are equivalent if they coincide on the real axis. A mapping $f$ is compatible with $\Gamma$ if $f^{-1}\circ \gamma \circ f\in PSL(2,\mathbb{R})$ for all $\gamma\in\Gamma$. The space of equivalence classes of $\Gamma$-compatible mappings is called the Teichm\"uller space of $\Gamma$ and is denoted by $T(\Gamma)$. The space $T(\Gamma)$ is isomorphic to a bounded complex domain in $\mathbb{C}^{3g-3+n+l}$. The Teichm\"uller modular group ${\rm Mod}(\Gamma)$ acts on $T(\Gamma)$ by complex isomorphisms. Denote by ${\rm Mod}_0(\Gamma)$ the subgroup of ${\rm Mod}(\Gamma)$ consisting of pure mapping classes (i.~e. those fixing the punctures and elliptic points on $X$ pointwise). The factor $T(\Gamma)/{\rm Mod}_0(\Gamma)$ is isomorphic to the moduli space $\mathcal{M}_{g,n+l}$ of smooth complex algebraic curves of genus $g$ with $n+l$ labeled points.
\begin{remark} Note that $T(\Gamma)$, as well as the quotient space $T(\Gamma)/{\rm Mod}(\Gamma)$, actually depends not on the signature of $\Gamma$, but rather on its \emph{signature type}, the unordered set $r=\{r_{1},r_{2},\dots\}$, 
where $r_1=n$ and $r_{i}$ is the number of elliptic points of order $i,\; i=2,3,\ldots$ (see \cite{Bers}).
\end{remark}

The holomorphic tangent and cotangent spaces to $T(\Gamma)$ at the origin are isomorphic to $\Omega^{-1,1}(X)$ and $\Omega^{2,0}(X)$ respectively (where, as before, 
$X=\Gamma\backslash\HH$). 
Let $B^{-1,1}(X)$ be the unit ball in $\Omega^{-1,1}(X)$ with respect to the $L^{\infty}$ norm and let $\beta:B^{-1,1}(X)\rightarrow T(\Gamma)$ be the Bers map. It defines complex coordinates in the neighborhood of the origin in $T(\Gamma)$ by the assignment
$$
(\varepsilon_{1},\dots,\varepsilon_{d})\mapsto \Gamma^{\mu}=f^{\mu}\circ\Gamma\circ (f^{\mu})^{-1},
$$ 
where $\mu=\varepsilon_{1}\mu_{1}+\cdots +\varepsilon_{d}\mu_{d}$, $\mu_{1},\dots,\mu_{d}$ is a basis for $\Omega^{-1,1}(X)$, 
and $f^{\mu}$ is a quasiconformal mapping of $\HH$ that fixes 0, 1, $\infty$ and satisfies the Beltrami equation
$$
f^{\mu}_{\bar{z}}=\mu f^{\mu}_{z}.
$$
For $\mu\in\Omega^{-1,1}(X)$ denote by $\dfrac{\del}{\del\varepsilon_\mu}$
and $\dfrac{\del}{\del\bar\varepsilon_\mu}$ the partial derivatives along the holomorphic curve $\beta(\varepsilon\mu)$ in $T(\Gamma)$, 
where $\varepsilon\in\mathbb{C}$ is a small parameter.

The Cauchy-Riemann operators $\bar\del_{k}$ form a holomorphic ${\rm Mod}(\Gamma)$-invariant family of operators on $T(\Gamma)$. 
The determinant bundle $\lambda_{k}$ 
associated with $\bar\del_{k}$ is a holomorphic ${\rm Mod}(\Gamma)$-invariant line bundle on $T(\Gamma)$ whose fibers are given by the determinant lines
$\wedge^{\rm max}\ker\bar\del_{k}\otimes(\wedge^{\rm max}{\rm coker}\,\bar\del_{k})^{-1}$. Since the kernel and cokernel of $\bar\del_k$ are the spaces of harmonic
differentials $\Omega^{k,0}(X)$ and $\Omega^{k,1}(X)$ respectively, the line bundle $\lambda_k$ is Hermitian with the metric induced by the Hodge
scalar products in the spaces $\Omega^{p,q}(X)$ 
(note that each orbifold Riemann surface $X=\Gamma\backslash\HH$ inherits a natural metric of constant negative curvature $-1$). The corresponding norm in $\lambda_k$ we will 
denote by $||\cdot||_k$. 
Note that by duality between $\Omega^{k,0}(X)$ and $\Omega^{1-k,1}(X)$ the determinant line bundles $\lambda_k$ and $\lambda_{1-k}$ are isometrically isomorphic.

The Quillen norm in $\lambda_k$ is defined by the formula 
\begin{align}
||\cdot||_k^Q=\dfrac{||\cdot||_k}{\sqrt{\det\Delta_k}}\;\label{qm}
\end{align}
for $k\leq 0$ and is extended for all $k$ by the isometry $\lambda_{k}\cong \lambda_{1-k}$.
The determinant $\det\Delta_{k}$ defined via the Selberg zeta fuction is a smooth ${\rm Mod}(\Gamma)$-invariant function on $T(\Gamma)$. 

\section{Main results} \label{3}
Our objective is to compute the canonical connection and the curvature (or the first Chern form) of the Hermitian holomorphic line bundle $\lambda_k$ 
on $T(\Gamma)$. By Remark 1 $\lambda_k$ can be thought of as holomorphic $\mathbb{Q}$-line bundle on the moduli space $T(\Gamma)/{\rm Mod}(\Gamma)$.

\subsection{Connection form on the determinant bundle}
We start with computing the connection form on the determinant line bundle $\lambda_{-k}$ for $k>0$ relative to the Quillen metric. The following result generalizes Lemma 3 in \cite{TZ}:

\begin{theorem}
For any integer $k\geq 0$ and $\mu\in\Omega^{-1,1}(X)$ we have
\begin{align}
&\dfrac{\del}{\del\vep_\mu}{\Big |}_{\vep_\mu=0}\log\det\Delta_{-k}=\nonumber\\
&=-\int_{X}\del y^{-2k}\del'\,\left(G_{-k}(z,z';1)-Q_{-k}(z,z';1)\right){\Big |}_{z'=z}\mu(z)\,d^2z\;,\label{1st}
\end{align}
where $\del=\dfrac{\del}{\del z},\;\del'=\dfrac{\del}{\del z'}$, and $d^2z=\dfrac{dz\wedge d\bar{z}}{-2\sqrt{-1}}$ is the Euclidean area form on $\HH$.
\end{theorem}
\begin{remark}\label{vp}
The integral in \eqref{1st} is absolutely convergent if $m_i>2$ for all $i=1,\ldots,l$. If $m_i=2$ for some $i$, then this integral should be understood in the principal value sense as follows. Let $z_i$ be the fixpoint of the elliptic generator $T_i\in\Gamma$ of order 2, and consider the mapping 
$h_i:\HH\rightarrow \D,\; h_i(z)=\dfrac{z-z_i}{z-\bar{z}_i}$. Denote by $B_\delta=\{u\in \D\,\big|\,|u|<\delta\}$ the disk of radius $\delta$ in $\D$ with center at 0. 
Since $\Gamma$ is discrete, for $\delta$ small enough we have
$h_i^{-1}(B_\delta)\cap\,\gamma h_j^{-1}(B_\delta)=\emptyset$ unless $i=j$ and $\gamma$ is either $I$ or $T_i$. The subset
\begin{align*}
\HH_\delta=\HH\setminus\big(\textstyle\bigcup_{\{i\,|\,m_i=2\}}\textstyle\bigcup_{\gamma\in\Gamma_i\backslash\Gamma}\,\gamma h_i^{-1}(B_\delta)\big)
\end{align*} 
is $\Gamma$-invariant, where $\Gamma_i$ denotes the cyclic group of order $2$ generated by $T_i$. The factor $X_\delta=\Gamma\backslash\HH_\delta$ is an orbifold Riemann surface
with holes centered at the conical points of angle $\pi$. The integral in the right hand side of \eqref{1st} we then define as
\begin{align*}
\lim_{\delta\to 0}\,\int_{X_\delta}\del y^{-2k}\del'\,\left(G_{-k}(z,z';1)-Q_{-k}(z,z';1)\right){\Big |}_{z'=z}\mu(z)\,d^2z\;.
\end{align*}
\end{remark}
\begin{proof}

We will use the results of \cite{TZ} profoundly. Repeating verbatim the proof of Lemma 3 in \cite{TZ} we get for $k\geq 0$
\begin{align}
\dfrac{\del}{\del\vep_{\mu}}&{\Big |}_{\vep_\mu=0}\log\det\Delta_{-k}=-\int\limits_{X}\del\del'\Big(G_{0}(z,z';k+1)-Q_{0}(z,z';k+1)-\nonumber\\
&-\sum_{\substack{\gamma\in\Gamma,\\ \gamma \;\mathrm{parabolic}}}Q_{0}(z,\gamma z';k+1)
-\sum_{\substack{\gamma\in\Gamma,\\ \gamma\;\mathrm{elliptic}}}Q_{0}(z,\gamma z';k+1)\Big)\Big|_{z'=z}\mu(z)d^{2}z\;.\label{vardet}
\end{align}
Note that by Lemma 3 in \cite{TZ} the contribution from parabolic elements to the right hand side of \eqref{vardet} vanishes, i.~e.
\begin{align*}
\int\limits_{X}\sum_{\substack{\gamma\in\Gamma,\\ \gamma \;\mathrm{parabolic}}}\del\del'Q_{0}(z,\gamma z';k+1)\Big|_{z'=z}\mu(z)d^{2}z=0\;.
\end{align*}
By Lemma 4 in \cite{TZ} we can further rewrite \eqref{vardet} as follows:
\begin{align*}
\dfrac{\del}{\del\vep_{\mu}}{\Big |}_{\vep_\mu=0}\log\det\Delta_{-k}=&-\int\limits_{X}\del y^{-2k}\del'\Big(G_{-k}(z,z';1)-Q_{-k}(z,z';1)\\
&-\sum_{\substack{\gamma\in\Gamma,\\ \gamma\;\mathrm{elliptic}}}Q_{-k}(z,\gamma z';1)\gamma'(\bar{z}')^{-k}\Big)\Big|_{z'=z}\mu(z)d^{2}z\;.
\end{align*}
The integrand in the right hand side is smooth and the integral is absolutely convergent, cf. \eqref{G}. We need to show that
\begin{align}
\int\limits_{\Gamma\backslash\HH}\Big(\sum_{\substack{\gamma\in\Gamma,\\ \gamma\;\mathrm{elliptic}}}\del y^{-2k}\del'Q_{-k}(z,\gamma z';1)\gamma'(\bar{z}')^{-k}\Big)\Big|_{z'=z}\mu(z)d^{2}z=0\label{ec}
\end{align}
(if there is $m_i=2$ we understand this integral as the principle value, see Remark \ref{vp}).

Without loss of generality we may assume that $l=1$ and $\Gamma$ has one elliptic generator $T$ of order $m$ with fixpoint $z_0\in \HH$. Then by \eqref{Q} we have
\begin{align*}
-\sum_{\substack{\gamma\in\Gamma,\\ \gamma\;\mathrm{elliptic}}}&\del y^{-2k}\del'Q_{-k}(z,\gamma z';1)\gamma'(\bar{z}')^{-k}\Big)\Big|_{z'=z}\\
& =\dfrac{1}{\pi}\sum_{\substack{\gamma\in\Gamma,\\ \gamma \;\mathrm{elliptic}}}\dfrac{1}{(z-\gamma z)^{2}}\left(\dfrac{\gamma z-\gamma\bar{z}}{\bar{z}-\gamma z}\right)^{2k}\gamma'(z)\gamma'(\bar{z})^{-k}\\
&=\dfrac{(z-\bar{z})^{2k}}{\pi}\sum_{\substack{\gamma\in\Gamma,\\ \gamma \;\mathrm{elliptic}}}\dfrac{\gamma'(z)^{k+1}}{(z-\gamma z)^{2}(\bar{z}-\gamma z)^{2k}}\\
&=\dfrac{(z-\bar{z})^{2k}}{\pi}\sum_{\sigma\in\Gamma_0\bk\Gamma}\;\sum_{i=1}^{m-1}
\dfrac{(\sigma^{-1} T^{\,i}\sigma)'(z)^{k+1}}{(z-\sigma^{-1} T^{\,i}\sigma z)^{2}(\bar{z}-\sigma^{-1} T^{\,i}\sigma z)^{2k}}\\
&=\dfrac{(z-\bar{z})^{2k}}{\pi}\sum_{\sigma\in\Gamma_0\bk\Gamma}\;\sum_{i=1}^{m-1}
\dfrac{\sigma'(z)\sigma'(\bar{z})^{k}( T^{\,i}\sigma)'(z)^{k+1}}{(\sigma z- T^{\,i}\sigma z)^{2}(\sigma\bar{z}- T^{\,i}\sigma z)^{2k}}\\
&=\dfrac{1}{\pi}\sum_{\sigma\in\Gamma_0\bk\Gamma}\;\sum_{i=1}^{m-1}
\dfrac{(\sigma z-\sigma \bar{z})^{2k}(T^{\,i})'(\sigma z)^{k+1}\sigma'(z)^{2}}{(\sigma z-T^{\,i}\sigma z)^{2}(\sigma\bar{z}-T^{\,i}\sigma z)^{2k}}\\
&=\dfrac{1}{\pi}\sum_{\sigma\in\Gamma_0\bk\Gamma}\phi(\sigma z)\sigma'(z)^2\;,
\end{align*}
where $\Gamma_0\cong\mathbb{Z}/m\mathbb{Z}$ is the cyclic group generated by $T$ (the stabilizer of $z_0$ in $\Gamma$), and
\begin{align*}
\phi(z)=\sum_{i=1}^{m-1}\dfrac{(z-\bar{z})^{2k}(T^{\,i})'(z)^{k+1}}{(z-T^{\,i} z)^{2}(\bar{z}-T^{\,i}z)^{2k}}
\end{align*}

Since $\phi(Tz)T'(z)^2=\phi(z)$, it is easy to check that the last expression in the above formula is a (meromorphic) quadratic differential on $X$. Using the standard substitution $u=\dfrac{z-z_0}{z-\bar{z}_0}$ we get
\begin{align}
\dfrac{(z-\bar{z})^{2k}(T^{\,i})'(z)^{k+1}}{(z-T^{\,i} z)^{2}(\bar{z}-T^{\,i} z)^{2k}}
=\dfrac{\omega^{i(k+1)}}{(1-\omega^{i})^{2}\,u^{2}}\cdot\dfrac{(1-|u|^{2})^{2k}}{(1-\omega^{i}|u|^{2})^{2k}}\left(\dfrac{du}{dz}\right)^{2}\;.\label{uz}
\end{align}
Since $\mu(0)=0$ for $m>2$, see \eqref{mu-disk}, the integral in the left hand side of \eqref{ec} is absolutely convergent, and we have
\begin{align*}
&\int\limits_{X}\Big(\sum_{\substack{\gamma\in\Gamma,\\ \gamma\;\mathrm{elliptic}}}\del y^{-2k}\del'Q_{-k}(z,\gamma z';1)\gamma'(\bar{z}')^{-k}\Big)\Big|_{z'=z}\mu(z)\,d^{2}z\\
&=\dfrac{1}{\pi}\sum_{i=1}^{m-1}\dfrac{\omega^{i(k+1)}}{(1-\omega^{i})^{2}}\int\limits_{\Gamma_0\backslash\D}\left(\dfrac{1-|u|^{2}}{1-\omega^{i}|u|^{2}}\right)^{2k}\dfrac{\mu(u)}{u^{2}}\,d^{2}u\\
&=\dfrac{1}{4\pi}\sum_{i=1}^{m-1}\dfrac{\omega^{i(k+1)}}{(1-\omega^{i})^{2}}\int\limits_{\Gamma_0\backslash\D}\left(\dfrac{1-|u|^{2}}{1-\omega^{i}|u|^{2}}\right)^{2k}(1-|u|^{2})^{2}\sum_{j=1}^{\infty}\bar{a}_{jm}\bar{u}^{jm-2}\dfrac{d^{2}u}{u^2}\\
&=\dfrac{1}{4\pi}\sum_{i=1}^{m-1}\dfrac{\omega^{i(k+1)}}{(1-\omega^{i})^{2}}\sum_{j=1}^{\infty}\bar{a}_{jm}\int_{0}^{\frac{2\pi}{m}}\int_{0}^{1}\dfrac{(1-r^{2})^{2k+2}}{(1-\omega^{i}r^{2})^{2k}}\,r^{jm-3}e^{\sqrt{-1}jm\theta}drd\theta\\
&=0,
\end{align*}
that proves the theorem for $m>2$ (in the last line we used polar coordinates $u=re^{\sqrt{-1}\theta}$ on $\D$).

We have to be more careful in the case $m=2$, since the contribution from elliptic elements is no longer absolutely convergent and should be considered as the principal value, see Remark \ref{vp}. From now on we assume that $\Gamma$ acts on the unit disk $\D$, so that $\Gamma_0$ is generated by $\omega=-1$. Since $\Gamma$ is discrete, there exists $\min_{\gamma\in\Gamma/\{\pm 1\},\;\gamma\neq\pm 1}|\gamma(0)|>0$. Therefore, we can choose a small $\delta$ such that
$B_\delta\cap\,\gamma B_\delta=\emptyset$ unless $\gamma=\pm 1$. The set $\D_\delta=\D\setminus(\cup_{\gamma\in\Gamma/\{\pm 1\}}\,\gamma B_\delta)$ is $\Gamma$-invariant, 
and the factor $X_\delta=\Gamma\backslash\D_\delta$ is a Riemann surface with a small hole centered at the conical point. In this case we have
\begin{align*}
&\int\limits_{X_\delta}\Big(\sum_{\substack{\gamma\in\Gamma,\\ \gamma\;\mathrm{elliptic}}}\del y^{-2k}\del'Q_{-k}(z,\gamma z';1)\gamma'(\bar{z}')^{-k}\Big)\Big|_{z'=z}\mu(z)\,d^{2}z\\
&=\dfrac{(-1)^{k+1}}{4\pi}\int\limits_{\D_\delta/\{\pm 1\}}\left(\dfrac{1-|u|^{2}}{1+|u|^{2}}\right)^{2k}\dfrac{\mu(u)}{u^{2}}\,d^{2}u\\
&=\dfrac{(-1)^{k+1}}{16\pi}\int\limits_{(\D\setminus B_\delta)/\{\pm 1\}}\left(\dfrac{1-|u|^{2}}{1+|u|^{2}}\right)^{2k}(1-|u|^{2})^{2}\sum_{j=1}^\infty\bar{a}_{jm}\bar{u}^{jm-2}\,\dfrac{d^{2}u}{u^2}\\
&\hspace{1in}-\dfrac{(-1)^{k+1}}{4\pi}\sum_{\substack{\gamma\in\Gamma/\{\pm 1\}\\\gamma\neq\pm 1}}\int_{\gamma B_\delta}\left(\dfrac{1-|u|^{2}}{1+|u|^{2}}\right)^{2k}\dfrac{\mu(u)}{u^{2}}\,d^{2}u\;.
\end{align*}
For the first integral in the last line we have
\begin{align*}
\int\limits_{(\D\setminus B_\delta)/\{\pm 1\}}\left(\dfrac{1-|u|^{2}}{1+|u|^{2}}\right)^{2k}(1-|u|^{2})^{2}\sum_{j=1}^\infty\bar{a}_{jm}\bar{u}^{jm-2}\,\dfrac{d^{2}u}{u^2}=0
\end{align*}
by the same reason as in the case $m>2$ (in polar coordinates $u=re^{\sqrt{-1}\theta}$ the integral over $\theta$ vanishes). As for the sum of integrals, since the integrand is uniformly bounded on $\D\setminus B_\delta$ and the (Euclidean) area of
the union $\cup_{\gamma\in\Gamma/\{\pm 1\},\;\gamma\neq\pm 1}\,\gamma B_\delta$ tends to 0 as $\delta\rightarrow 0$ we have
\begin{align*}
\sum_{\substack{\gamma\in\Gamma/\{\pm 1\}\\\gamma\neq\pm 1}}\int_{\gamma B_\delta}\left(\dfrac{1-|u|^{2}}{1+|u|^{2}}\right)^{2k}\dfrac{\mu(u)}{u^{2}}\,d^{2}u
\underset{\delta\to 0}{\xrightarrow{\hspace*{1cm}}}0\,, 
\end{align*}
which proves the theorem.
\end{proof}

Later we will need to know the behavior of the quadratic differential
\begin{align*}
R_{-k}(z)=-\del y^{-2k}\del'\,\left(G_{-k}(z,z';1)-Q_{-k}(z,z';1)\right){\Big |}_{z'=z}
\end{align*}
near the elliptic fixpoints of $\Gamma$. Let $T$ be an elliptic generator of $\Gamma$ of order $m$ with fixpoint $z_0$. The standard isomorphism $\HH\to \D$ given by
$u=\frac{z-z_0}{z-\bar{z}_0}$ maps $z_0\in \HH$ to $0\in D$, so that $T$ becomes the multiplication by $\omega=e^{2\pi\sqrt{-1}/m}$. Slightly abusing notation, we put $R_{-k}(u)du^2=R_{-k}(z)dz^2$. Then we have
\begin{lemma}\label{asymp}
The quadratic differential $R_{-k}$ on $\D$ has the following asymtotics as $u\to 0$:
\begin{align}
R_{-k}(u) & =-\dfrac{m^2}{2\pi}\left(B_2\left(\left\{\dfrac{k}{m}\right\}\right)-\dfrac{1}{6m^2}\right)\,\dfrac{1}{u^2}+O(1)\label{0}\\
\intertext{and}
\frac{\del}{\del\bar{u}}R_{-k}(u) & =-\frac{2km}{\pi}\left(B_{1}\left(\left\{\frac{k}{m}\right\}\right)+\frac{1}{2m}\right)\frac{1}{u} +O(1),\label{0-1}
\end{align}
where $B_{1}(x)=x-\frac{1}{2}$ and $B_2(x)=x^2-x+\tfrac{1}{6}$ are Bernoulli polynomials, and $\{x\}$ denotes the fractional part of $x\in\Q$.
\end{lemma}
\begin{proof}
Using \eqref{uz} we easily see that
\begin{align}
R_{-k}(u)&=\dfrac{1}{\pi}\sum_{i=1}^{m-1}\dfrac{\omega^{i(k+1)}}{(1-\omega^{i})^{2}\,u^{2}}\cdot\dfrac{(1-|u|^{2})^{2k}}{(1-\omega^{i}|u|^{2})^{2k}}+O(1) \nonumber\\
&=\dfrac{1}{\pi}\sum_{i=1}^{m-1}\dfrac{\omega^{i(k+1)}}{(1-\omega^{i})^{2}}\cdot\dfrac{1}{u^{2}}+O(1)\quad {\rm as}\; u\to 0.\label{R-k-0}
\end{align}
We are going to show now that
\begin{align}
\sum_{i=1}^{m-1}\dfrac{\omega^{i(k+1)}}{(1-\omega^{i})^{2}}=-\dfrac{m^{2}-1}{12} + \dfrac{\bar{k}(m-\bar{k})}{2}\;\label{iii},
\end{align}
where $\bar{k}$ is the least nonnegative residue of $k$ modulo $m$.
We start with the simple identity
$$\sum_{i=1}^{m-1}\log(x-\omega^{i})=\log(1+x+\cdots+x^{m-1})\;.$$
Differentiating it once with respect to $x$ and putting $x=1$ we get
\begin{align}
\sum_{i=1}^{m-1}\dfrac{1}{1-\omega^{i}}=\dfrac{m-1}{2}\;,\label{i}
\end{align} 
Differentiating it twice, putting $x=1$ and applying \eqref{i} we get 
\begin{align}
\sum_{i=1}^{m-1}\dfrac{\omega^{i}}{(1-\omega^{i})^{2}}=-\dfrac{m^{2}-1}{12}\;.\label{ii}
\end{align}
To prove \eqref{iii} we use the identity
$$\dfrac{x^{\bar{k}+1}}{(1-x)^{2}}=\dfrac{x}{(1-x)^{2}}-\dfrac{\bar{k}}{1-x}+\sum_{j=0}^{\bar{k}-1}(\bar{k}-j)x^{j}$$
together with (\ref{i}) and (\ref{ii}) to obtain 
\begin{align*}
\sum_{i=1}^{m-1}\dfrac{\omega^{i(k+1)}}{(1-\omega^{i})^{2}} & =-\dfrac{m^{2}-1}{12}-\dfrac{\bar{k}(m-1)}{2}+\sum_{j=0}^{\bar{k}-1}(\bar{k}-j)\sum_{i=1}^{m-1}\omega^{ij}\\
 & =-\dfrac{m^{2}-1}{12}-\dfrac{\bar{k}(m-1)}{2}+\bar{k}(m-1) -\sum_{j=1}^{\bar{k}-1}(\bar{k}-j)\\
 &=-\dfrac{m^{2}-1}{12} + \dfrac{\bar{k}(m-\bar{k})}{2}\\
 &=-\dfrac{m^2}{2}\left(B_2\left(\left\{\dfrac{k}{m}\right\}\right)-\dfrac{1}{6m^2}\right)\,,
\end{align*}
which proves \eqref{0}. 

To prove \eqref{0-1}, we differentiate  \eqref{R-k-0} to obtain
$$\frac{\del}{\del\bar{u}}R_{-k}(u) = \dfrac{2k}{\pi}\sum_{i=1}^{m-1}\dfrac{\omega^{i(k+1)}}{\omega^{i}-1}\cdot\dfrac{1}{u}+O(1)\quad {\rm as}\; u\to 0,$$
and use  elementary identity
$$\sum_{i=1}^{m-1}\dfrac{\omega^{i(k+1)}}{\omega^{i}-1}=-m\left(B_{1}\left(\left\{\frac{k}{m}\right\}\right)+\frac{1}{2m}\right),$$
which is proved in the same way as \eqref{iii}.
\end{proof}

\subsection{The first Chern form}\label{3.2}
Our next objective is to compute the curvature, or the first Chern form $c_1(\lambda_{-k},\,||\cdot||_{-k}^Q$), of the determinant line bundle $\lambda_{-k}$ endowed with Quillen's metric, see \eqref{qm}. 
To formulate the theorem we introduce three kinds of metrics:
\begin{itemize}
\item {\em Weil-Petersson metric.} For $\mu,\,\nu\in\Omega^{-1,1}(X)$ understood as tangent vectors to the Teichm\"{u}ller space $T(\Gamma)$, the Weil-Petersson scalar product is defined by the formula
\begin{align}
\left\langle\dfrac{\del}{\del\vep_\mu},\dfrac{\del}{\del\vep_\nu}\right\rangle_{{\rm WP}}=\int_{X}\mu(z)\overline{\nu (z)}d\rho(z)\;,\label{WP}
\end{align}
where $d\rho$ is the hyperbolic area form on $X=\Gamma\backslash\HH$. This metric is K\"ahler, and its symplectic form will be denoted by $\omega_{{\rm WP}}$.
\item {\em Cuspidal metric} (also known as Takhtajan-Zograf metric). For the parabolic generator $S_i$ of $\Gamma$ this metric is defined as
\begin{align}
\left\langle\dfrac{\del}{\del\vep_\mu},\dfrac{\del}{\del\vep_\nu}\right\rangle_{i}^{\rm cusp}=\int_{X}E_i(z,2)\mu(z)\overline{\nu (z)}d\rho(z)\;,\label{TZ}
\end{align}
where $E_i(z,s)$ is the $i$-th Eisenstein series for $\Gamma$. By definition,
\begin{align}
E_i(z,s)=\sum_{\gamma\in\langle S_i\rangle\backslash\Gamma}{\rm Im}(\sigma_i^{-1}\gamma z)^s,\qquad i=1,\ldots,n\,,\label{eis}
\end{align}
where $\langle S_i\rangle$ denotes the cyclic subroup of $\Gamma$ generated by $S_i$, and $(\sigma_i^{-1}S_i\sigma) z=z\pm 1$.
The series is absolutely convergent for ${\rm Re}\,s>1$, is positive for $s=2$ and satisfies the equation 
$$
\Delta_0E_i(z,s)=\tfrac{1}{4}s(1-s)E_i(z,s).
$$ 
For any $i=1,\ldots,n$ this metric is K\"ahler, its symplectic form we denote by $\omega_i^{\rm cusp}$
and put $\omega_{\rm cusp}=\sum_{i=1}^n\omega_i^{\rm cusp}$.
\item {\em Elliptic metric.} For the elliptic generator $T_j$ of $\Gamma$ define
\begin{align}
\left\langle\dfrac{\del}{\del\vep_\mu},\dfrac{\del}{\del\vep_\nu}\right\rangle_{j}^{\rm ell}=\int_{X}G(z_j,z)\mu(z)\overline{\nu (z)}d\rho(z)\;,\label{ell}
\end{align}
where $z_j$ is the fixpoint of $T_j$, and $G(z,z')=G_{0}(z,z';2)$ is the integral kernel of $\left(\Delta_0+\tfrac{1}{2}\right)^{-1}$. As we will see later, the metrics $\langle\;,\;\rangle_j^{\rm ell}$ are also K\"ahler. Denote by $\omega_j^{\rm ell}$ the $(1,1)$-form
$$
\omega_{j}^{\rm ell}\left(\dfrac{\del}{\del\vep_\mu},\dfrac{\del}{\del\vep_\nu}\right)= 
-\frac{1}{2}\,{\rm Im}\left\langle\dfrac{\del}{\del\vep_\mu},\dfrac{\del}{\del\vep_\nu}\right\rangle_{j}^{\rm ell}.
$$
\end{itemize}

The main result of this paper is
\begin{theorem}\label{theorem 2}
For integer $k\geq 0$ we have
\begin{align}\label{main}
&c_1(\lambda_{-k}, ||\cdot||_{-k}^Q)=\dfrac{6k^2+6k+1}{12\pi^2}\omega_{\rm WP}-\dfrac{1}{9}\omega_{\rm cusp}\nonumber\\
&+\dfrac{1}{4\pi}\sum_{j=1}^l \left(2k B_{1}\left(\left\{\frac{k}{m_{j}}\right\}\right) -
B_2\left(\left\{\dfrac{k}{m_j}\right\}\right) +\dfrac{6k-1}{6m_j}\right)\,\omega_j^{\rm ell}\;,
\end{align}
where, as above, $B_{1}(x)=x-\frac{1}{2}$ and $B_2(x)=x^2-x+\tfrac{1}{6}$ are Bernoulli polynomials, and $\{x\}$ denotes the fractional part of $x$. 
\end{theorem}
\begin{remark}\label{isom}
This result holds for $k<0$ as well, because the Hermitian line bundles $\big(\lambda_k\,,||\cdot||_k^Q\big)$ and $\big(\lambda_{1-k}\,,||\cdot||_{1-k}^Q\big)$
are isometrically isomorphic.
\end{remark}
\begin{proof}
As before, without loss of generality we assume that $l=1$ and $\Gamma$ has one elliptic generator $T$ of order $m$ with fixpoint $z_0\in H$.
We start with \eqref{1st}, where for $m=2$ we understand the integral in the right hand side as the principle value as described above. We have
\begin{align}
&\dfrac{\del^2}{\del\vep_\mu\del\bar\vep_\nu}{\Big |}_{\vep_\mu=\vep_\nu=0}\log\det\Delta_{-k}=\nonumber\\
&=-\dfrac{\del}{\del\bar\vep}{\Big |}_{\vep=0} \int\limits_{X^{\vep\nu}}\del y^{-2k}\del'\,\left(G^{\vep\nu}_{-k}(z,z';1)-Q_{-k}(z,z';1)\right){\Big |}_{z'=z}\mu^{\vep\nu}(z)\,d^2z\nonumber\\
&=\dfrac{\del}{\del\bar\vep}{\Big |}_{\vep=0} \int\limits_{X^{\vep\nu}}R^{\vep\nu}_{-k}(z)\mu^{\vep\nu}(z)\,d^2z\,.\label{2nd}
\end{align}
Here we use the following notation: 
\begin{itemize}
\item[$\circ$] $\Gamma^{\vep\nu}=f^{\vep\nu}\circ\Gamma\circ(f^{\vep\nu})^{-1}$, where $f^{\vep\nu}:\mathbb{H}\to \mathbb{H}$ is the Fuchsian deformation satisfying the Beltrami equation $$f^{\vep\nu}_{\bar z}=\vep\nu f^{\vep\nu}_z$$ 
and fixing 0, 1 and $\infty$, 
\item[$\circ$] $G^{\vep\nu}_{-k}(z,z';1)$ is the Green's function of $\Delta_{-k}$ on
$X^{\vep\nu}=\Gamma^{\vep\nu}\backslash\HH,\; k>0$, whereas for $k=0$ the Green's function is understood as the constant term in the Laurent expansion of
$G_{0}(z,z';s)$ at $s=1$, see \eqref{G_0}, 
\item[$\circ$] $\mu^{\vep\nu}\in\Omega^{-1,1}(X^{\vep\nu})$ is the parallel transport of $\mu\in\Omega^{-1,1}(X)$ along the trajectory of the tangent vector $\vep\nu$, and
\item[$\circ$] $R^{\vep\nu}_{-k}$ is a quadratic differential on $X^{\vep\nu}\setminus \{f^{\vep\nu}(z_0)\}$ given by the formula
\begin{align*}
R^{\vep\nu}_{-k}(z)=-\del y^{-2k}\del'\,\left(G^{\vep\nu}_{-k}(z,z';1)-Q_{-k}(z,z';1)\right){\Big |}_{z'=z}\,.
\end{align*}
\end{itemize}
For $\varphi^{\vep\nu}\in C^{p,q}(X^{\vep\nu}\setminus \{f^{\vep\nu}(z_0)\})$, we define its pullback $(f^{\vep\nu})^*\varphi^{\vep\nu}$ to $X\setminus \{z_0\}$ by the formula
$$
(f^{\vep\nu})^*\varphi^{\vep\nu}=\varphi^{\vep\nu}\circ f^{\vep\nu} (f^{\vep\nu}_z)^p (f^{\vep\nu}_{\bar z})^q\in C^{p,q}(X\setminus \{z_0\})\,,
$$
where $C^{p,q}(X\setminus \{z_0\})$ denotes the space of smooth $(p,q)$-differentials on the punctured at the elliptic point $z_0$ orbisurface $X\setminus \{z_0\}$.
Let $\Gamma_0$ denote the stabilizer of $z_0$ in $H$ generated by $T$, 
and put $(X^{\vep\nu})_\delta=(X^{\vep\nu})\setminus h_{\vep\nu}^{-1}(B_\delta)$, 
where $h_{\vep\nu}: \HH\to \D,\;h_{\vep\nu}(z)=\dfrac{z-f^{\vep\nu}(z_0)}{z-\overline{f^{\vep\nu}(z_0)}}$, and $B_\delta$ is a disk of small radius $\delta$ in the unit disk $\D$. Using Ahlfors' lemma
$$
\dfrac{\del}{\del\bar\vep}{\Big |}_{\vep=0}\dfrac{|f^{\vep\nu}_z|^2}{({\rm Im}\,f^{\vep\nu})^2}=0
$$
for $\nu\in\Omega^{-1,1}(X)$ (see \cite{A}), we continue \eqref{2nd} as follows:
\begin{align}
&\dfrac{\del^2}{\del\vep_\mu\del\bar\vep_\nu}{\Big |}_{\vep_\mu=\vep_\nu=0}\log\det\Delta_{-k}=\nonumber\\
&\qquad =\dfrac{\del}{\del\bar\vep}{\Big |}_{\vep=0} \left(\lim_{\delta\to 0}\int\limits_{X^{\vep\nu}_\delta}R^{\vep\nu}_{-k}(z)\mu^{\vep\nu}(z)\,d^2z\right)\nonumber\\
&\qquad =\dfrac{\del}{\del\bar\vep}{\Big |}_{\vep=0} \left(\lim_{\delta\to 0}\int\limits_{(f^{\vep\nu})^{-1}(X^{\vep\nu}_\delta)}(f^{\vep\nu})^*R^{\vep\nu}_{-k}(z)(f^{\vep\nu})^*\mu^{\vep\nu}(z)\,d^2z\right)\nonumber\\
&\qquad =\lim_{\delta\to 0}\int\limits_{X_\delta}\left(\dfrac{\del}{\del\bar\vep}{\Big |}_{\vep=0}(f^{\vep\nu})^*R^{\vep\nu}_{-k}(z)\right)\mu(z)\,d^2z\nonumber\\
&\qquad +\lim_{\delta\to 0}\int\limits_{X_\delta}R_{-k}(z)\left(\dfrac{\del}{\del\bar\vep}{\Big |}_{\vep=0}(f^{\vep\nu})^*\mu^{\vep\nu}(z)\right)\,d^2z\nonumber\\
&\qquad -\dfrac{\sqrt{-1}}{2}\lim_{\delta\to 0}\int\limits_{\del X_\delta}R_{-k}(z)\mu(z)\left(\dfrac{\del}{\del\bar\vep}{\Big |}_{\vep=0}f^{\vep\nu}d\bar{z}-\dfrac{\del}{\del\bar\vep}{\Big |}_{\vep=0}\overline{f^{\vep\nu}}dz\right)\nonumber\\
&\qquad =I_1+I_2+I_3\,,\label{3I}
\end{align}
where the integral $I_{3}$ is due to the variation of the domain of integration $(f^{\vep\nu})^{-1}(X^{\vep\nu}_\delta)$ in \eqref{2nd}.

The first integral in the right hand side of \eqref{3I} was computed in \cite{TZ}, Theorem 1, Formulas (4.7) and (4.8):
\begin{align}
\label{I1}
I_1={\rm Tr}((-\mu\bar\nu\,I+(\del_\mu\bar\del_{-k})\Delta_{-k}^{-1}(\del_{\bar\nu}\bar\del_{-k}^*))P_{-k,1})+\dfrac{3k+1}{12\pi}\langle\mu,\nu\rangle_{\rm WP}\,,
\end{align}
where $I$ is the identity operator in the Hilbert space $\mathcal{H}^{-k,1}(X)$, $P_{-k,1}:\mathcal{H}^{-k,1}(X)\rightarrow \Omega^{-k,1}(X)$ is the orthogonal projector, ${\rm Tr}$ is the trace, and
\begin{align*}
&\del_\mu\bar\del_{-k}=\dfrac{\del}{\del\vep}{\Big |}_{\vep=0}(f^{\vep\mu})^*\,\bar\del_{-k} ((f^{\vep\mu})^*)^{-1},\\
&\del_{\bar\nu}\bar\del_{-k}^*=\dfrac{\del}{\del\bar\vep}{\Big |}_{\vep=0}(f^{\vep\nu})^*\,\bar\del_{-k}^* ((f^{\vep\nu})^*)^{-1}.
\end{align*}

We proceed with the integral $I_2$ in the right hand side of \eqref{3I}. We will use Wolpert's formula \cite{W}
\begin{align*}
\dfrac{\del}{\del\bar\vep}{\Big |}_{\vep=0}(f^{\vep\nu})^*\mu^{\vep\nu}(z)=-\dfrac{\del}{\del\bar{z}}\, y^2 \dfrac{\del}{\del\bar{z}}f_{\mu\bar\nu},
\end{align*}
where, as before, $f_{\mu\bar\nu}=\left(\Delta_0+\tfrac{1}{2}\right)^{-1}(\mu\bar{\nu})$. Then by Stokes' theorem
\begin{align}
I_2=&-\lim_{\delta\to 0}\int\limits_{X_\delta} R_{-k}(z)\,\dfrac{\del}{\del\bar{z}}\left( y^2 \dfrac{\del}{\del\bar{z}}f_{\mu\bar\nu}(z) \right) d^2z \nonumber\\
=&\int\limits_{X}\dfrac{\del }{\del\bar{z}}R_{-k}(z)\,\dfrac{\del }{\del\bar{z}}f_{\mu\bar\nu}(z) y^2 d^2z \nonumber\\
&+\dfrac{\sqrt{-1}}{2}\lim_{Y\to\infty}\int\limits_{\del X^Y} R_{-k}(z)\,\dfrac{\del }{\del\bar{z}}f_{\mu\bar\nu}(z) y^2 dz\nonumber\\
&+\dfrac{\sqrt{-1}}{2}\lim_{\delta\to 0}\int\limits_{\del X_\delta} R_{-k}(z)\,\dfrac{\del }{\del\bar{z}}f_{\mu\bar\nu}(z) y^2 dz\nonumber\\
=& I_4+I_5+I_6\,,\label{6I}
\end{align}
where $X^Y$ denotes the Riemann surface $\Gamma\backslash\HH$ with cusps cut off along horocycles at level $Y$ (see \cite{TZ} for details). The first two integrals in the right hand side of \eqref{6I} were computed in \cite{TZ}, Theorem 1. Namely, in case $l=0$ (no elliptic elements) it was shown that the integral $I_{4}$ is equal to
\begin{align*}
I_4=& -k\,{\rm Tr}\left(y^2\left(\dfrac{\del^2}{\del\vep_\mu\del\bar\vep_\nu}\Big|_{\vep_\mu=\vep_\nu=0}(f^{\vep_\mu\mu+\vep_\nu\nu})^*\left(y^{-2}\right)\right)\,P_{-k,1}\right)
\nonumber \\
&+\dfrac{k(2k+1)}{4\pi}\langle\mu,\nu\rangle_{\rm WP}
\end{align*}
and
\begin{align}\label{I-5}
I_5=-\dfrac{\pi}{9}\langle\mu,\nu\rangle_{\rm cusp}\;.
\end{align}

In case of elliptic elements the computation of the integral $I_{5}$ in \cite{TZ} is unchanged and the formula \eqref{I-5} remains valid. However, the computation of the integral $I_{4}$ uses 
integration by parts, so one needs to take into the account the singularity of $\del R_{-k}/\del\bar{z}$ at $z=z_{0}$. This results in the extra contribution $I_{40}$ to the integral $I_{4}$, which we write using the coordinate $u$ in the unit disk $\D$. 

Namely, put  $C_\delta=\{u=\delta e^{\sqrt{-1}\theta}\,|\,0\leq\theta\leq\frac{2\pi}{m}\}$ (note that its orientation is opposite to the orientation of $\del X_{\delta}$) and denote for brevity
$$b(m,k)=-2km\left(B_{1}\left(\left\{\frac{k}{m}\right\}\right)+\frac{1}{2m}\right).$$
 Using asymptotic \eqref{0-1} in Lemma \ref{asymp}, this contribution can be written as
 \begin{align}
I_{40} & = \frac{\sqrt{-1}}{2}\lim_{\delta\to 0}\int_{C_{\delta}}\frac{\del}{\del\bar{u}}R_{-k}(u)f_{\mu\bar\nu}(u)\rho(u)^{-1}du \nonumber\\
& = \frac{\sqrt{-1}}{2\pi}\lim_{\delta\to 0}\int_{C_{\delta}}\frac{b(m,k)}{u}\dfrac{(1-|u|^2)^2}{4} f_{\mu\bar\nu}(u)du\nonumber\\
& =-\dfrac{b(m,k)}{8\pi} \lim_{r\to 0}\int_0^{\frac{2\pi}{m}} (1-r^2)^{2} f_{\mu\bar\nu}(r,\theta)d\theta\nonumber\\
&=-\dfrac{b(m,k)}{4m}c_{0}\nonumber\\
&=\dfrac{k}{2}\left(B_1\left(\left\{\dfrac{k}{m}\right\}\right)+\dfrac{1}{2m}\right)\int\limits_{X}G(0,u)\mu(u)\overline{\nu(u)}d\rho(u),\label{I4-0}
 \end{align}
where we used the Fourier expansion for $f_{\mu\bar\nu}$ and Lemma \ref{f-0}. 

The integral $I_{6}$ is computed similarly.
Denoting for brevity
\begin{align*}
c(m,k)=-\dfrac{m^{2}-1}{12} + \dfrac{\bar{k}(m-\bar{k})}{2}=-\dfrac{m^2}{2}\left(B_2\left(\left\{\dfrac{k}{m}\right\}\right)-\dfrac{1}{6m^2}\right)
\end{align*}
and using Lemma \ref{asymp}, we have
\begin{align}
I_6&=-\dfrac{\sqrt{-1}}{2\pi} \lim_{\delta\to 0}\int_{C_\delta}\dfrac{c(m,k)}{u^2}\,\dfrac{(1-|u|^2)^2}{4}\,\dfrac{\del f_{\mu\bar\nu}}{\del\bar{u}}\,du \nonumber\\
&=\dfrac{c(m,k)}{16\pi} \lim_{r\to 0}\int_0^{\frac{2\pi}{m}} \dfrac{(1-r^2)^2}{r}\left(\dfrac{\del}{\del r}+\dfrac{\sqrt{-1}}{r}\dfrac{\del}{\del\theta}\right) f_{\mu\bar\nu}(r,\theta)d\theta \nonumber\\
&=\dfrac{c(m,k)}{16\pi} \lim_{r\to 0}\int_0^{\frac{2\pi}{m}} \dfrac{1}{r}\,\dfrac{\del f_{\mu\bar\nu}}{\del r}d\theta \nonumber\\
&=\dfrac{c(m,k)}{8m} \lim_{r\to 0}\dfrac{1}{r}\,\dfrac{\del f_0}{\del r}\nonumber\\
&=\dfrac{c(m,k)}{4m}\,c_2\nonumber\\
&=-\dfrac{m}{4}\left(B_2\left(\left\{\dfrac{k}{m}\right\}\right)-\dfrac{1}{6m^2}\right)\left(\int\limits_{X}G(0,u)\mu(u)\overline{\nu(u)}d\rho(u)
-2\mu(0)\overline{\nu(0)}\right)\,,
\end{align}
where we used the Fourier expansion for $f_{\mu\bar\nu}$ and Lemma \ref{f-0}. (Note that the term $-2\mu(0)\overline{\nu(0)}$ is present only when $m=2$.)

The only integral that is left to compute is $I_3$ in \eqref{3I}. As in the case of the integral $I_6$, we evaluate $I_3$ using the coordinate $u$ in the unit disk $\mathbb{D}$. We have
\begin{align*}
I_3 & =\dfrac{\sqrt{-1}}{2}\lim_{\delta\to 0}\int\limits_{C_\delta}R_{-k}(u)\mu(u)\left(\dfrac{\del}{\del\bar\vep}{\Big |}_{\vep=0}f^{\vep\nu}d\bar{u}-\dfrac{\del}{\del\bar\vep}{\Big |}_{\vep=0}\overline{f^{\vep\nu}}du\right)\\
& =\dfrac{\sqrt{-1}}{2}\lim_{\delta\to 0}\int\limits_{C_\delta}R_{-k}(u)\mu(u)(\Phi(u)d\bar{u}-\overline{F(u)}du)\,,
\end{align*}
where we put 
$$\Phi=\dfrac{\del}{\del\bar\vep}{\Big |}_{\vep=0}f^{\vep\nu}\quad\text{and}\quad F= \dfrac{\del}{\del\vep}{\Big |}_{\vep=0}f^{\vep\nu}.$$ 
Since for $m>2$ we have $\mu(0)=0$, cf. \eqref{mu-disk}, this yields $I_{3}=0$. When $m=2$, we use Lemma \ref{asymp}, the fact that $\Phi$ is holomorphic \cite{A}, and the formulas
$$\mu_{u}(0)=0, \quad F(u)=F(0)+\nu(0)\bar{u}+F_{u}(0)u +O(u^{2})$$
and \eqref{0} to obtain
\begin{align*}
I_3 & =\dfrac{c(2,k)}{2\pi}\lim_{r\to 0}\int_{0}^{\pi}\!\frac{\mu(re^{\sqrt{-1}\theta})}{r}\!\left(e^{-3\sqrt{-1}\theta}\,\Phi(re^{\sqrt{-1}\theta})+e^{-\sqrt{-1}\theta}\,\overline{F(re^{\sqrt{-1}\theta})}\right)\!d\theta \\
& =\dfrac{c(2,k)}{2}\mu(0)\bar{\nu}(0)\,,
\end{align*}
where we put $u=re^{\sqrt{-1}\theta}$.
Thus, for all $m\geq 2$
$$I_{3}+I_{6}=-\dfrac{m}{4}\left(B_2\left(\left\{\dfrac{k}{m}\right\}\right)-\dfrac{1}{6m^2}\right)\int\limits_{X}G(0,u)\mu(u)\overline{\nu(u)}d\rho(u). $$

\allowdisplaybreaks[0]
To complete the proof we recall Lemma 1 in \cite{ZT} (or Lemma 5 in \cite{TZ}) that computes the curvature (or the first Chern form) of the determinant line bundle $\lambda_{-k}$ relative to the standard $L^2$-metric $||\cdot||_{-k}$:
\begin{align}
&c_1(\lambda_{-k}, ||\cdot||_{-k})(\mu,\nu)\nonumber\\
&\qquad=\dfrac{\sqrt{-1}}{2\pi}\,{\rm Tr}\Big(\Big(\Big(\mu\bar\nu + ky^2\dfrac{\del^2}{\del\vep_\mu\del\bar\vep_\nu}\Big|_{\vep_\mu=\vep_\nu=0}(f^{\vep_\mu\mu+\vep_\nu\nu})^*\left(y^{-2}\right)\Big)I\nonumber\\
&\qquad\quad - (\del_\mu\bar\del_{-k})\Delta_{-k}^{-1}(\del_{\bar\nu}\bar\del_{-k}^*)\Big)P_{-k,1}\Big)\,.\label{c1}
\end{align}
Here we use the same notation as in formula \eqref{I1}, and $\mu,\nu\in\Omega^{-1,1}(X)$ are understood as tangent vectors to $T(\Gamma)$ at the origin. 
Then for the first Chern form of $\lambda_{-k}$ relative to the Quillen metric we have
\begin{align}
c_1(\lambda_{-k}, ||\cdot||_{-k}^Q)(\mu,\nu)&=c_1(\lambda_{-k}, ||\cdot||_{-k})(\mu,\nu)\nonumber\\
&+\dfrac{\sqrt{-1}}{2\pi}\dfrac{\del^2}{\del\vep_\mu\del\bar\vep_\nu}{\Big |}_{\vep_\mu=\vep_\nu=0}\log\det\Delta_{-k}\,.\label{c1_Q}
\end{align}
Substituting formulas \eqref{3I} -- \eqref{c1} into \eqref{c1_Q}, we arrive at the assertion of the theorem.
\end{proof}

\section{Concluding remarks} \label{4}

\subsection{Local potential for elliptic metric} \label{potential}
Let $\Gamma$ be a cofinite Fuchsian group, and let $T$ be an elliptic generator of $\Gamma$ of order $m$ with the fixpoint $0\in\D$. Following \cite{PTT}, 
we are going to show that the positive definite Hermitian product \eqref{ell}
\begin{align*} 
\left\langle\dfrac{\del}{\del\vep_\mu},\dfrac{\del}{\del\vep_\nu}\right\rangle^{\rm ell}=\int_{X}G(0,u)\mu(u)\overline{\nu (u)}d\rho(u)\;,
\end{align*}
where $X=\Gamma\backslash\D$, has a local potential in a neighborhood of the origin in the Teichm\"{u}ller space $T(\Gamma)$.

For the sake of simplicity let us assume that the group $\Gamma$ has genus $0$. Let 
$$
J:\D\longrightarrow\C\smallsetminus\{w_{1},\dots,w_{n+l-3},0,1\}
\subset\Gamma\backslash\D
$$ 
be the corresponding {\em Hauptmodul} with ramification index $m$ over $J(0)=0$, 
where $n$ and $l$ are the numbers of parabolic and elliptic generators of $\Gamma$ respectively. 
The function $J$ has a power series expansion in $u\in\D$ of the form
$$
J(u)=\sum_{k=1}^{\infty}J_k\,u^{mk},
$$
where $J_1\neq 0$. For the density of the hyperbolic metric $e^{\varphi(w)}|dw|^2$ on $X=C\setminus\{w_{1},\dots,w_{n+l-3},0,1\}$ we have
$$
e^{\varphi(w)}=\dfrac{4|J^{-1}(w)'|^{2}}{(1-|J^{-1}(w)|^{2})^{2}}
$$
and
$$
\dfrac{4|J_1|^{-\frac{2}{m}}}{m^{2}}=\lim_{w\rightarrow 0}e^{\varphi(w)}|w|^{2-\frac{2}{m}}\,.
$$
Take $\mu\in\Omega^{-1,1}(X)$ and denote by $F^{\vep\mu}:\C\rightarrow\C$ the quasiconformal map satisfying the Beltrami equation 
$F^{\vep\mu}_{\bar{w}}=\vep\mu F^{\vep\mu}_w$ that fix 0, 1 and $\infty$. Let $\Gamma^{\vep\mu}=F^{\vep\mu}\circ\Gamma\circ (F^{\vep\mu})^{-1}$ be the deformation 
of the group $\Gamma$ in $T(\Gamma)$ in the direction of $\mu$. Then we can think of $F^{\vep\mu}$ as a map
$C\setminus\{w_{1},\dots,w_{n+l-3},0,1\}\rightarrow C\setminus\{w_{1}^{\vep\mu},\dots,w_{n+l-3}^{\vep\mu},0,1\}\in X^{\vep\mu}$,
where $X^{\vep\mu}=\Gamma^{\vep\mu}\backslash\D$ and $F^{\vep\mu}(w_i)=w_i^{\vep\mu}$.
Let us now put
$$
h^{\vep\mu}=-\log|J_1^{\vep\mu}|^{\frac{2}{m}}=2\log m-2\log 2+\lim_{w\rightarrow 0}\left(\varphi^{\vep\mu}(w)+\left(1-m^{-1}\right)\log|w|^{2}\right)\,.
$$
Then using Wolpert's formula \cite{W} for the second variation of the hyperbolic area form and the fact that $F^{\vep\mu}(w)$ is holomorphic in $\vep$,
we get 
\begin{align*}
\left.\dfrac{\del^{2}}{\del\vep\del\bar\vep}\right|_{\vep=0}h^{\vep\mu} &=\dfrac{1}{2}(\Delta_{0}+\tfrac{1}{2})^{-1}(|\mu|^{2})(0) \\
& = \dfrac{1}{2}\iint\limits_{X}G(0,u)|\mu(u)|^{2}d\rho(u)\\
& =\dfrac{1}{2}\left<\dfrac{\del}{\del\vep_\mu},\dfrac{\del}{\del\vep_\mu}\right>^{\rm ell}\,.
\end{align*}
In other words, $h^{\vep\mu}$ is a potential of the elliptic metric $\langle\;,\;\rangle^{\rm ell}$ that is defined globally on $\mathcal{M}_{0,n+l}$ for any
elliptic generator $T_1,\ldots,T_l$ of $\Gamma$.

If the group $\Gamma$ has genus $g>0$, one can use the Schottky uniformization to construct local potentials for the elliptic metrics in exactly the same way
(see \cite{PTT} for details). Thus, we have the following

\begin{theorem}
Let $\Gamma$ be a finitely generated cofinite Fuchsian group of signature $(g;n;m_1,\ldots,m_l)$. Then each Hermitian metric $\langle\;,\;\rangle^{\rm ell}_1,\ldots,\langle\;,\;\rangle^{\rm ell}_l$ 
defined by \eqref{ell} is K\"{a}hler on the Teichm\"{u}ller space $T(\Gamma)$ (or on the moduli space $T(\Gamma)/{\rm Mod}(\Gamma)$ in the orbifold sense).
\end{theorem}

As in the case of punctured Riemann surfaces in \cite{W2,PTT}, for each conical point $z_j$ we consider the tautological line bundle $\mathcal{L}_j$ on 
$T(\Gamma)$, or rather a $\mathbb{Q}$-line bundle on $T(\Gamma)/{\rm Mod}(\Gamma)$. Its fibers are holomorphic cotangent lines at conical points. 
Then as in \cite{PTT} (cf. also \cite{FP}) one can show that $h$ determines a Hermitian metric in the line bundle $\mathcal{L}$ and
$$
c_{1}(\mathcal{L}_j, h)=-\frac{1}{2\pi}\omega_j^{\rm ell}.
$$

\subsection{Cuspidal and elliptic metrics}\label{relation}
Here we will show that when the order of the elliptic generator tends to $\infty$, the corresponding Hermitian product converges to the cuspidal one.
Consider the family of elliptic transformations $T_m$ of order $m$ of the form
$$
T_m=C_m O_{m}C_m^{-1},
$$
where 
\begin{align*}
O_{m}=\begin{pmatrix} \;\;\,\cos\dfrac{2\pi}{m} &\sin\dfrac{2\pi}{m}\\ \;&\;\\ -\sin\dfrac{2\pi}{m} &  \cos\dfrac{2\pi}{m} \end{pmatrix}\,,
\qquad C_{m}=\begin{pmatrix} 0 & \sqrt{\dfrac{m\vphantom{2}}{2\pi}}\\-\sqrt{\dfrac{2\pi}{m}} & 0\end{pmatrix}\,,
\end{align*}
and $m=2,3,\ldots$ Then
$$
T_{m}=\begin{pmatrix} \cos\dfrac{2\pi}{m} & \dfrac{m}{2\pi}\sin\dfrac{2\pi}{m} \\\;&\;\\
-\dfrac{2\pi}{m}\sin\dfrac{2\pi}{m} & \cos\dfrac{2\pi}{m}\end{pmatrix} \longrightarrow \begin{pmatrix} 1 & 1\\0 & 1\end{pmatrix} 
$$
as $m\rightarrow\infty$, and $\zeta_m=\dfrac{\sqrt{-1}m}{2\pi}$ (the fixpoint of $T_m$ in $\HH$) tends to $\sqrt{-1}\,\infty$.

To compute the limit of the elliptic scalar product as $m\rightarrow\infty$ we use Fay's formula \cite[Theorem 3.1]{Fay} 
$$
G_0(z,z';s)=\dfrac{4\,y^{1-s}}{2s-1}E(z',s)+O(e^{-2\pi y})
$$
as $y\rightarrow\infty$ and $y>y'$. Here $G_0(z,z';s)$ stands, as before, for the integral kernel of $\left(\Delta_{0}+\dfrac{s(s-1)}{4}\right)^{-1}$ on $X=\Gamma\backslash\HH$, and $E(z,s)$ is the Eisenstein series associated with the parabolic subgroup generated by $\left(\begin{smallmatrix} 1&1\\0&1 \end{smallmatrix}\right)$. 
Putting $s=2$ we get for $m$ large that
$$
G(\zeta_{m},z)=G_0(\zeta_m,z;2)=\dfrac{8\pi}{3m}E(z,2)+O(e^{-m}).
$$
Thus we see that 
$$
\dfrac{3m}{8\pi}\left<\dfrac{\del}{\del\vep_\mu},\dfrac{\del}{\del\vep_\nu}\right>^{\rm ell}_m\longrightarrow
\left<\dfrac{\del}{\del\vep_\mu},\dfrac{\del}{\del\vep_\nu}\right>^{\rm cusp}_\infty\quad {\rm as}\quad m\rightarrow\infty\,,
$$
where $\langle\;,\;\rangle^{\rm ell}_m$ is the Hermitian product associated with the elliptic generator $T_m$, and 
$\langle\;,\;\rangle^{\rm cusp}_\infty$ is the Hermitian product associated with the parabolic generator $\left(\begin{smallmatrix} 1&1\\0&1 \end{smallmatrix}\right)$\,.

\subsection{Elliptic metric and Selberg zeta values}\label{zeta values}
Here we give a simple example of a relation between the elliptic metric and Selberg zeta values considered as functions on the Teichm\"uller space $T(\Gamma)$.
As $\Gamma$ we take a Fuchsian group of the first kind of signature $(0;1;2,2,2)$, i.~e. 
$$
\Gamma=\{S_0,T_1,T_2,T_3\,|\,S_0T_1T_2T_3=T_1^2=T_2^2=T_3^2=I\}.
$$
Let $\chi:\Gamma\to\Z/2\Z$ be the character defined on the generators by $\chi(S_0)=\chi(T_1)=\chi(T_2)=\chi(T_3)=-1$, and let $\Gamma'=\ker\chi$.
Then $\Gamma'$ is a torsion-free subgroup of $\Gamma$ of index 2 and signature (1;1) given by
$$
\Gamma'=\{A_1,B_1,S_1\,|\,A_1B_1A_1^{-1}B_1^{-1}S_1=I\},
$$
where $A_1=T_1T_2,\;A_2=T_3T_2$ and $S_1=S_0^2$. The group $\Gamma'$ uniformizes a once punctured elliptic curve given by a lattice $\Lambda=\Z\cdot 1+\Z\cdot\tau\subset\C$ with ${\rm Im}\tau>0$, so that $\Gamma'\backslash\HH\simeq\Lambda\backslash\C-\{0\}$. The Teichm\"uller spaces of $\Gamma$ and $\Gamma'$ are naturally isomorphic: $T(\Gamma)=T(\Gamma')=\{\tau\in\C\,|\,{\rm Im}\tau>0\}$. Formula \eqref{main} applied to the determinant line bundles $\lambda_k$ and 
$\lambda'_k$ on $T(\Gamma)$ and $T(\Gamma')$ respectively yields
\begin{align}
&c_1(\lambda_k,||\cdot||_k^Q)=\dfrac{6k^2-6k+1}{12\pi^2}\omega^{}_{\mathrm{WP}}-\dfrac{1}{9}\omega^{}_{\rm cusp}+\dfrac{(-1)^k}{16\pi}\omega^{}_{\rm ell}\,,\label{f1}\\
&c_1(\lambda'_k,||\cdot||_k^Q)=\dfrac{6k^2-6k+1}{12\pi^2}\omega'_{\mathrm{WP}}-\dfrac{1}{9}\omega'_{\rm cusp}\label{f2}
\end{align}
(here we assume that $k\geq 1$, and $\omega^{}_{\rm ell}=\omega_1^{\rm ell}+\omega_2^{\rm ell}+\omega_3^{\rm ell}$). 

Since the fundamental domain of $\Gamma'$ is twice the fundamental domain of $\Gamma$, and the Beltrami differential corresponding to $\del/\del\tau$ is the same for both $\Gamma$ and $\Gamma'$, we have $\omega'_{\mathrm{WP}}=2\omega^{}_{\mathrm{WP}}$. Moreover, since $\langle S_1\rangle\backslash\Gamma'=\langle S_0\rangle\backslash\Gamma$, the Eisenstein series for $\Gamma$ and $\Gamma'$ are equal, i.~e. $E(z,s;\Gamma)=E(z,s;\Gamma')$, see \eqref{eis}. Therefore, we have
$\omega'_{\mathrm{cusp}}=2\omega^{}_{\mathrm{cusp}}$, and comparing \eqref{f1} and \eqref{f2} we see that
\begin{align}
2c_1(\lambda_k,||\cdot||_k^Q)-c_1(\lambda'_k,||\cdot||_k^Q)=\dfrac{(-1)^k}{8\pi}\omega^{}_{\rm ell}\,.\label{diff}
\end{align}
For $k=1$ we have
\begin{align*}
&c_1(\lambda_1,||\cdot||_1^Q)=\dfrac{\sqrt{-1}}{2\pi}\bar\del_\tau\del_\tau\log\left(\dfrac{1}{Z'(1,\Gamma,1)}\right)\,,\\
&c_1(\lambda'_1,||\cdot||_1^Q)=\dfrac{\sqrt{-1}}{2\pi}\bar\del_\tau\del_\tau\log\left(\dfrac{{\rm Im}\tau}{Z'(1,\Gamma',1)}\right)\,,
\end{align*}
where $\del_\tau$ and $\bar\del_\tau$ are the (1,0)- and (0,1)-components of the exterior derivative operator on the upper half-plane $\{\tau\in\C\,|\,{\rm Im}\tau>0\}$, and $Z(s,\Gamma,\chi)$ is defined by \eqref{zeta}. By \cite[Theorem 3.1]{VZ} we have 
$$
Z(s,\Gamma',1)=Z(s,\Gamma,1)\cdot Z(s,\Gamma,\chi)
$$ 
and hence $Z'(1,\Gamma',1)=Z'(1,\Gamma,1)\cdot Z(1,\Gamma,\chi)$ (note that $Z(1,\Gamma,\chi)\neq 0$). Substituting this expression for $Z'(1,\Gamma',1)$ into \eqref{diff}, we finally obtain that for a group $\Gamma$ of signature $(0;1;2,2,2)$
\begin{align}
\sqrt{-1}\,\omega^{}_{\rm ell}=-\dfrac{d\tau\wedge d\bar\tau}{({\rm Im}\tau)^2}+4\,\bar\del_\tau\del_\tau\log\left(\dfrac{Z(1,\Gamma,\chi)}{Z'(1,\Gamma,1)}\right)
\end{align}
on $T(\Gamma)=\{\tau\in\C\,|\,{\rm Im}\tau>0\}$. Here $\Gamma_\tau=f^{\mu}\circ\Gamma\circ(f^{\mu})^{-1}$, where $f^{\mu}:\mathbb{H}\to \mathbb{H}$ is the Fuchsian deformation satisfying the Beltrami equation $f^{\mu}_{\bar z}=\mu f^{\mu}_z$ with $\mu\in\Omega^{-1,1}(\Gamma\backslash\HH)$ corresponding to the tangent vector $\del/\del\tau$.

\end{document}